\documentclass{amsart}
\usepackage{amsmath,amssymb, amsthm, amscd}
\input xy
\xyoption{all}

\newtheorem{theorem}{Theorem}

\newtheorem{definition}[theorem]{Definition}
\newtheorem{lemma}[theorem]{Lemma}
\newtheorem{proposition}[theorem]{Proposition}
\theoremstyle{remark}

\newcommand{\tr}{\operatorname{tr}}
\newcommand{\GL}{\operatorname{GL}}
\newcommand{\K}[1]{K^{\operatorname{#1}}}
\newcommand{\M}{\operatorname{M}}
\newcommand{\R}{\operatorname{R}}
\renewcommand{\Re}{\operatorname{Re}}
\renewcommand{\Im}{\operatorname{Im}}

\newcommand{\SC}{\mathcal{SC}}

\newcommand{\abs}[1]{\left\vert #1\right\vert}
\newcommand{\norm}[1]{{\left\vert\kern-0.25ex\left\vert #1\right\vert\kern-0.25ex\right\vert}}
\newcommand{\normJ}[1]{{\left\vert\kern-0.25ex\left\vert #1\right\vert\kern-0.25ex\right\vert}_J}
\newcommand{\normA}[1]{{\left\vert\kern-0.25ex\left\vert #1\right\vert\kern-0.25ex\right\vert}_A}
\newcommand{\normAt}[1]{{\left\vert\kern-0.25ex\left\vert #1\right\vert\kern-0.25ex\right\vert}_{\widetilde{A}}}
\newcommand{\normnA}[1]{{\left\vert\kern-0.25ex\left\vert #1\right\vert\kern-0.25ex\right\vert}_{\M(n, A)}}
\newcommand{\normnJ}[1]{{\left\vert\kern-0.25ex\left\vert #1\right\vert\kern-0.25ex\right\vert}_{\M(n, J)}}
\newcommand{\normJt}[1]{{\left\vert\kern-0.25ex\left\vert #1\right\vert\kern-0.25ex\right\vert}_{\widetilde{J}}}
\newcommand{\normCone}[1]{{\left\vert\kern-0.25ex\left\vert #1\right\vert\kern-0.25ex\right\vert}_{C^1}}
\newcommand{\normLtwo}[1]{{\left\vert\kern-0.25ex\left\vert #1\right\vert\kern-0.25ex\right\vert}_{L^2(X\times\mathbb{R})}}
\newcommand{\normSCone}[1]{{\left\vert\kern-0.25ex\left\vert #1\right\vert\kern-0.25ex\right\vert}_{\mathcal{SC}^1}}
\newcommand{\normTone}[1]{{\left\vert\kern-0.25ex\left\vert #1\right\vert\kern-0.25ex\right\vert}_{\mathcal{T}^1}}
\newcommand{\normellone}[1]{{\left\vert\kern-0.25ex\left\vert #1\right\vert\kern-0.25ex\right\vert}_{\ell^1(\mathbb{R}_d,\delta)}}
\newcommand{\normellonem}[1]{{\left\vert\kern-0.25ex\left\vert #1\right\vert\kern-0.25ex\right\vert}_{\ell^1(\mathbb{R}_d,m)}}
\newcommand{\normone}[1]{{\left\vert\kern-0.25ex\left\vert #1\right\vert\kern-0.25ex\right\vert}_1}
\newcommand{\normoplus}[1]{{\left\vert\kern-0.25ex\left\vert #1\right\vert\kern-0.25ex\right\vert}_\oplus}
\newcommand{\normp}[1]{{\left\vert\kern-0.25ex\left\vert #1\right\vert\kern-0.25ex\right\vert}_p}
\newcommand{\normtwo}[1]{{\left\vert\kern-0.25ex\left\vert #1\right\vert\kern-0.25ex\right\vert}_2}
\newcommand{\normop}[1]{{\left\vert\kern-0.25ex\left\vert #1\right\vert\kern-0.25ex\right\vert}_{op}}
\newcommand{\normtrace}[1]{{\left\vert\kern-0.25ex\left\vert #1\right\vert\kern-0.25ex\right\vert}_{L^1(\tau)}}
\newcommand{\normsup}[1]{{\left\vert\kern-0.25ex\left\vert #1\right\vert\kern-0.25ex\right\vert}_\infty}
\newcommand{\vertiii}[1]{{\left\vert\kern-0.25ex\left\vert\kern-0.25ex\left\vert #1 
    \right\vert\kern-0.25ex\right\vert\kern-0.25ex\right\vert}}

\begin{document}
\title[Determinant formula for Toeplitz operators]{A determinant formula for Toeplitz operators associated to a minimal flow}

\author{Efton Park}
\address{Department of Mathematics, Texas Christian University, Box 298900, Fort Worth, TX, 76129, United States}
\email{e.park@tcu.edu}

\begin{abstract}
We define a determinant on the Toeplitz algebra associated to a minimal flow, give a formula for this determinant in terms of symbols, and show that
this determinant can be used to give information about the algebraic $K$-theory of functions on the underlying space.
\end{abstract}

\maketitle



Let $A$ be an $n \times n$ matrix with complex entries.  The equation
\[
\det(e^A) = e^{\tr(A)} \tag{$\star$}
\]
establishes a beautiful and important relationship between the trace and the determinant on $\M(n, \mathbb{C})$.  When $\mathcal{H}$ is an infinite-dimensional
separable Hilbert space, we can extend the definition of trace on $\M(n, \mathbb{C})$ to produce a trace on the algebra
$\mathcal{B}(\mathcal{H})$ of bounded linear maps on $\mathcal{H}$, although this trace is only defined on the so-called ideal of trace class
operators.  We can then use functional calculus and $(\star)$ to define a determinant on invertible elements of $\mathcal{B}(\mathcal{H})$ of the form
1 + trace class.  

For other operator algebras, the way forward is much less clear.  Suppose that $\mathcal{M}$ is a type $\operatorname{II}_1$ factor endowed with the
normal trace $\tau$ for which $\tau(1) = 1$.  In \cite{FK}, Fuglede and Kadison defined a determinant on $\GL(1, \mathcal{M})$ via the formula
\[
{\det}^{FK}_\tau(x) = \exp\left(\tau\left(\log\left(\abs{x}\right)\right)\right). 
\]
The Fuglede-Kadison determinant is a group homomorphism from $\GL(1, \mathcal{M})$ onto the multiplicative group $\mathbb{R}^+$ 
of positive real numbers; this requires a significant amount of effort to prove.  This determinant provides valuable information about invertible
elements of $\mathcal{M}$.  For example, Fack and de la Harpe (\cite{FH}, Proposition 2.5)
proved that the kernel of the Fuglede-Kadison determinant is precisely the commutator subgroup 
$\left[\GL(1, \mathcal{M}), \GL(1, \mathcal{M})\right]$ of $\GL(1, \mathcal{M})$, and 
$\det^{FK}_\tau$ induces an isomorphism 
\[
\frac{\GL(1, \mathcal{M})}{\left[\GL(1, \mathcal{M}), \GL(1, \mathcal{M})\right]} \cong \mathbb{R}^+.
\]

Brown constructed a version of the Fuglede-Kadison determinant for semi-finite von Neumann algebras (\cite{B2}).   As is the case for
the Fuglede-Kadison determinant, it is nontrivial to prove that the Brown determinant satisfies the desired algebraic properties of a determinant.  

Suppose that $A$ is a Banach algebra with unit $1$ equipped with a continuous complex trace $\tau$, and 
let $\underline{\tau}$ denote the complex-valued map that $\tau$ induces
on $K$-theory.  In \cite{HS}, de la Harpe and Skandalis defined a group homomorphism
\[
\Delta_\tau: \GL(A)_0 \rightarrow \frac{\mathbb{C}}{\underline{\tau}(\operatorname{K}_0(A))}.
\]
This map is called the de la Harpe-Skandalis determinant, but it would more accurately described as the log of a 
determinant, because the group operation on $\mathbb{C}\slash\underline{\tau}(\operatorname{K}_0(A))$ is addition.  
For a $\operatorname{II}_1$ factor $\mathcal{M}$, we have $\GL(\mathcal{M})_0 = \GL(\mathcal{M})$ and 
$\underline{\tau}(\operatorname{K}_0(\mathcal{M})) = \mathbb{R}$, and 
\[
{\det}^{FK}_\tau(x) = \exp\left(\Re(2\pi i\Delta_\tau(x))\right)
\]
for all $x$ in $\GL(1, \mathcal{M})$.

In \cite{HKS}, Hochs, Kaad, and Schemaitat used Connes and Karoubi's long exact sequence relating operator algebra $K$-theory,
algebraic $K$-theory and ``relative'' $K$-theory (\cite{Kar}, \cite{CK}) to extend de la Harpe and Skandalis's determinant to the situation where
$A$ is a unital Banach algebra with a trace $\tau$ defined on a not necessarily closed ideal $J$ in $A$.
If $A$ is a semifinite von Neumann algebra $\mathcal{N}$ and $J$ is the trace ideal $L^1(\mathcal{N})$ associated to a semifinite trace $\tau$ on 
$\mathcal{N}$, the Hochs-Kaad-Schemaitat determinant is a homomorphism
\[
{\det}_\tau:  \K{alg}_1(\mathcal{N}, L^1(\mathcal{N})) \longrightarrow (0, \infty).
\]
If $x$ is in $\GL(L^1(\mathcal{N}))$, then $x$ determines an element $[x]$ of $\K{alg}_1(\mathcal{N}, L^1(\mathcal{N})) $, and $\det_\tau([x])$
agrees with Brown's determinant applied to $x$.  
 An advantage of the approach Hochs, Kaad, Schemaitat take is that the standard properties of a determinant follow directly
 from algebraic properties of the groups and maps described in \cite{Kar} and \cite{CK}.  
 
For operators that arise in a geometric or topological setting, one might hope to compute determinants in terms of geometric or topological data.  For example,
suppose $T_\phi$ and $T_\psi$ are Toeplitz operators on the circle with smooth symbols $\phi$ and $\psi$, and suppose $\phi$ and $\psi$
are non-vanishing with winding number zero.  Then $T_\phi$ and $T_\psi$ are invertible, and the multiplicative commutator 
\[
T_\phi T_\psi T_\phi^{-1} T_\psi^{-1} = I + (T_\phi T_\psi - T_\psi T_\phi) T_\phi^{-1} T_\psi^{-1}
\]
has the form identity plus a trace class operator, and therefore has a well-defined determinant.  The determinant of 
$T_\phi T_\psi T_\phi^{-1} T_\psi^{-1}$ can be computed in various ways 
from the symbols $\phi$ and $\psi$; for see example, Helton and Howe (\cite{HH}, Proposition 10.1); Brown (\cite{B1}, Section 6);
and Carey and Pincus (\cite{CP}, Proposition 1).  In fact, $\det(T_\phi T_\psi T_\phi^{-1} T_\psi^{-1})$ only
depends on the Steinberg symbol $\{\phi, \psi\}$ in the algebraic $K$-theory group $\K{alg}_2(C^\infty(S^1))$, and thus the determinant gives a way to detect 
nontrivial elements of $\K{alg}_2(C^\infty(S^1))$.  

In this paper, we consider the following generalization.  Let $X$ be a separable compact Hausdorff space equipped with a Borel measure
$\mu$ with full support and a minimal flow $\alpha$ such that $\alpha$ is measure-preserving and ergodic.   Then 
 $L^\infty(X) \rtimes \mathbb{R}$ is a type $\operatorname{II}_\infty$ factor.  Given $\phi$ in $C(X)$, define $M_\phi$
on $L^2(X \times \mathbb{R})$ via pointwise multiplication.  We can use the Hilbert transform to construct a Hardy space $H^2(X \times \mathbb{R})$
and then define a Toeplitz operator $T_\phi$ by compressing $M_\phi$ to $H^2(X \times \mathbb{R})$.  Let $\mathcal{T}(X, \alpha)$ be the
$C^*$-algebra generated by these Toeplitz operators, and let $\mathcal{SC}(X, \alpha)$ be the $C^*$-ideal generated by the semicommutators
$T_\phi T_\psi - T_{\phi\psi}$.  In \cite{CMX}, Curto, Muhly, and Xia constructed a short exact sequence 
\[
\xymatrix@C=30pt
{0 \ar[r]^-{} & \SC(X, \alpha) \ar[r]^-{} & \mathcal{T}(X, \alpha) \ar[r]^-{\sigma}  & C(X) \ar[r]^-{} & 0}
\]
such that $\sigma$ is a symbol map; that is, $\sigma(T_\phi) = \phi$ for all $\phi$ in $C(X)$.
Let $C^1(X, \alpha)$ denote the set of elements of $C(X)$ that are differentiable along the orbits of $\alpha$.   We will construct a 
\lq\lq sub-short exact sequence\rq\rq
\[
\xymatrix@C=30pt
{0 \ar[r]^-{} & \SC^1(X, \alpha) \ar[r]^-{} & \mathcal{T}^1(X, \alpha) \ar[r]^-{\sigma}  & C^1(X, \alpha) \ar[r]^-{} & 0}
\]
of Banach algebras with the property that elements of $\SC^1(X, \alpha)$ are trace class with respect to a semifinite trace $\tau$ on
 $L^\infty(X) \rtimes \mathbb{R}$.  We will then apply the Hochs-Kaad-Schemaitat machinery to this exact sequence to define a function
 $d$ on the unitization $\SC^1(X, \alpha)^+$ of $\SC^1(X, \alpha)$ that satisfies the properties of a determinant, and we produce a formula
 for $d$ on certain multiplicative commutators of elements of $\mathcal{T}^1(X, \alpha)$ in terms of symbols:
\[
d(e^Te^We^{-T}e^{-W}) =  \exp\left(\frac{1}{2\pi}\int_X{\kern-0.5ex}\Im\bigl(\sigma(T)^\prime(x)\sigma(W)(x)\bigr)\, d\mu(x)\right).
\]
 We will then employ the long exact sequence in algebraic $K$-theory to construct a group homomorphism from $\K{alg}_2(C^1(X, \alpha))$
 to $\mathbb{R}$; this map gives a way to show various elements of the highly complicated group $\K{alg}_2(C^1(X, \alpha))$ are nontrivial.
 
{\bf Acknowledgement:} I thank the anonymous referee of an earlier version of this paper for pointing out Curto, Muhly, and Xia's trace formula
for Toeplitz operators on minimal flows and suggesting the exploration of the connection between determinants and algebraic $K$-theory in this 
context.

\section{Toeplitz operators on minimal ergodic flows}\label{Toeplitz}

In this section we collect the results we need from \cite{CMX}, although our notation will differ
somewhat from theirs.   Let $X$ be a separable compact Hausdorff space equipped with a minimal flow $\alpha = \{\alpha_t\}_{t\in\mathbb{R}}$; given a point $x$ in $X$ and
a real number $t$, we will write $\alpha_t(x)$ as $x + t$.  Suppose $X$ admits a Borel probability measure $\mu$ with the following properties: (1) the support of $\mu$
 is all of $X$; (2) the maps $\alpha_t$ are measure-preserving for each real number $t$; and (3) $\alpha$ is ergodic with respect to $\mu$.
 
Endow $\mathbb{R}$ with Lebesgue measure and consider the Hilbert space $L^2(X \times \mathbb{R})$ associated
with the product measure on $X \times \mathbb{R}$.  Given $\phi$ in $C(X)$, define $M_\phi$ on $L^2(X \times \mathbb{R})$ by
pointwise multiplication:
\[
(M_\phi h)(x, s) = \phi(x)h(x, s).
\]
Define the Hilbert transform $H$ on $L^2(X \times \mathbb{R})$ by
\[
(Hf)(x, t) = \operatorname{PV}\left(\frac{1}{\pi i}\int_{-\infty}^\infty\frac{1}{s}f(x + s, t - s)\, ds\right).
\]
The factor of $\frac{1}{i}$, which is often not present in the definition of the Hilbert transform, is included in \cite{CMX} to make $H$ self adjoint.
Set $P = \frac{1}{2}(I + H)$.  Then $P$ is a projection; denote the range of $P$ by $H^2(X \times \mathbb{R})$.
 For each $\phi$ in $C(X)$, define the Toeplitz operator
$T_\phi: H^2(X \times \mathbb{R}) \rightarrow H^2(X \times \mathbb{R})$ by the formula
\[
T_\phi = PM_\phi.
\]
\begin{definition} The Toeplitz algebra associated to the flow $\alpha$ on $X$ is the $C^*$-subalgebra $\mathcal{T}(X, \alpha)$ of 
$\mathcal{B}(H^2(X \times \mathbb{R}))$ generated by the set $\{T_\phi : \phi \in C(X)\}$. 
\end{definition}

\begin{definition} The semi-commutator ideal of $\mathcal{T}(X, \alpha)$ is the $C^*$-ideal $\SC(X, \alpha)$ of $\mathcal{T}(X, \alpha)$ generated
by the set $\{T_\phi T_\psi - T_{\phi\psi} : \phi, \psi \in C(X)\}$.
\end{definition}

By Lemma 24.2 in \cite{CMX}, we have a short exact sequence
\[
\xymatrix@C=30pt
{0 \ar[r]^-{} & \SC(X, \alpha) \ar[r]^-{} & \mathcal{T}(X, \alpha) \ar[r]^-{\sigma}  & C(X) \ar[r]^-{} & 0}
\]
with the feature that $\sigma(T_\phi) = \phi$ for every $\phi$ in $C(X)$.  The short exact sequence has an isometric linear
splitting $\xi$ defined by $\xi(\phi) = T_\phi$.  As a consequence, every element of $\mathcal{T}(X, \alpha)$ can be uniquely
written in the form $T_\phi + S$ for some $\phi$ in $C(X)$ and $S$ in $\SC(X, \alpha)$, and $\normop{T_\phi} = \normsup{\phi}$
for every $\phi$ in $C(X)$.
\vskip 6pt
\emph{Remark:} The commutator ideal of $\mathcal{T}(X, \alpha)$ is contained in $\SC(X, \alpha)$, and   
Lemma 24.1 in \cite{CMX} states that if the action of $\alpha$ on $X$ is strictly ergodic, then these two ideals are equal.  This might
be true in general; I am not aware of any examples of minimal ergodic flows where the two ideals differ.  But in any case, we will not
investigate this issue here.
\vskip 6pt
For each real number $t$, define a unitary operator $U_t$ on $L^2(X \times \mathbb{R})$ by the formula
\[
(U_t h)(x, s) = h(x + t, t - s).
\]
Let $L^\infty(X) \rtimes \mathbb{R}$ be the von Neumann subalgebra of $\mathcal{B}(L^2(X \times \mathbb{R}))$
generated by the $M_\phi$ and $U_t$.  Because the action $\alpha$ is ergodic with respect to the measure $\mu$ and 
$\mu$ has full support,  $L^\infty(X) \rtimes \mathbb{R}$ is a type $\operatorname{II}_\infty$ factor 
and therefore admits a semifinite normal trace $\tau$.  The algebra $C_c(X \times \mathbb{R})$
is weakly dense in $L^\infty(X) \rtimes \mathbb{R}$; we scale $\tau$ so that 
\[
\tau(f) = \int_X f(x, 0)\, d\mu(x)
\]
for every $f$ in $C_c(X \times \mathbb{R})$.  
Following \cite{Kun}, define
\[
L^p(\tau) = \{F \in L^\infty(X) \rtimes \mathbb{R}: \tau\left(\abs{F}^p\right) < \infty\},\quad p = 1, 2
\]
and set 
\[
\normp{S} = \bigl(\tau\left(\abs{S}^p\right)\bigr)^{1\slash p}, \quad S \in L^p(\tau).
\]
(In fact, \cite{Kun} defines $L^p(\tau)$ for all $1 \leq p \leq \infty$, but $p = 1, 2$ are all we will need.)
Each $L^p(\tau)$ is an ideal in $L^\infty(X) \rtimes \mathbb{R}$.

\begin{proposition}[Holder's inequality (\cite{Kun}, Theorem 1)]\label{Holder} If $A$ and $B$ are in $L^2(\tau)$, then $AB$ is in $L^1(\tau)$, and
$\normone{AB} \leq \normtwo{A}\normtwo{B}$.
\end{proposition}

\begin{proposition}[(\cite{Kun}, Corollary 1.1)]\label{opineq}For all $S$ in $L^p(\tau)$ and $F$ in $L^\infty(X) \rtimes \mathbb{R}$,
we have the inequalities 
\[
\normp{SF} \leq \normp{S}\normop{F},\quad \normp{FS} \leq \normop{F}\normp{S}.
\]
\end{proposition}

We can decompose $L^2(X \times \mathbb{R})$ as 
$H^2(X \times \mathbb{R}) \oplus H^2(X \times \mathbb{R})^\perp$, and via this decomposition, we can
view $\mathcal{B}(H^2(X \times \mathbb{R}))$ as a subalgebra of $\mathcal{B}(L^2(X \times \mathbb{R}))$:
\[
S \hookrightarrow \begin{pmatrix} S & 0 \\ 0 & 0 \end{pmatrix}.
\]
Let $\mathcal{N}$ be the $\operatorname{II}_\infty$ factor
\[
\mathcal{N} = P\left(L^\infty(X) \rtimes \mathbb{R}\right)P. 
\]
The trace $\tau$ on $L^\infty(X) \rtimes \mathbb{R}$ restricts to $\mathcal{N}$, and we will use $\tau$ to denote the trace in both situations.  It should be clear from context
which is the appropriate domain for $\tau$.

\begin{definition} A function $\phi$ in $C(X)$ is differentiable with respect to $\alpha$ if the limit
\[
\phi^\prime(x) = \lim_{t\to0}\frac{\phi(x + t) - \phi(x)}{t}
\]
exists for each $x$ in $X$.  Let $C^1(X, \alpha)$ be the set of functions on $X$ that are continuously differentiable with respect to $\alpha$, and define a 
norm $\normCone{\cdot}$ on $C^1(X, \alpha)$ by the formula
\[
\normCone{\phi} = \normsup{\phi} + \normsup{\phi^\prime}.
\]
\end{definition}
It is routine to check that $C^1(X, \alpha)$ is a Banach algebra when equipped with this norm.

\begin{theorem}\label{semi} The semicommutator $T_\phi T_\psi - T_{\phi\psi}$ is in $L^1(\tau)$ for all $\phi$ and $\psi$ in $C^1(X, \alpha)$, and 
\[
\normone{T_\phi T_\psi - T_{\phi\psi}} \leq \normCone{\phi}\normCone{\psi}.
\]
\end{theorem}

\begin{proof} We adapt an argument from \cite{CMX}, Theorem 23.1.  Let $Q = 1 - P$ be orthogonal projection of 
$L^2(X \times \mathbb{R})$ onto $H^2(X \times \mathbb{R})^\perp$, and for each $\phi$ in $C(X)$, write $M_\phi$ as the matrix
\[
M_\phi = 
\begin{pmatrix}
PM_\phi P & PM_\phi Q \\ QM_\phi P & QM_\phi Q
\end{pmatrix}.
\]
Write $M_\phi M_\psi$ and $M_{\phi\psi}$ in this form.  Because these two operators are equal, the $(1,1)$ entries of their matrices agree, which
gives us the equation
\[
PM_{\phi\psi}P - (PM_\phi P)(PM_\psi  P) = (PM_\phi Q)(QM_\psi P).
\]
An easy computation shows that $PM_\phi Q = \frac{1}{2}P[H, M_\phi]$.  We also have
\[
QM_\psi P = (PM_{\bar{\psi}} Q)^* = \frac{1}{2}(P[H, M_{\bar{\psi}}])^* = -\frac{1}{2}P[H, M_\psi].
\]
By Lemma 22.1 in \cite{CMX}, the quantity
$[H, M_\phi]$ is in $L^2(\tau)$ for all $\phi$ in $C^1(X, \alpha)$.  Therefore $PM_\phi Q$ and $QM_\psi P$ are in $L^2(\tau)$, and thus their
product is in $L^1(\tau)$.  By identifying $\mathcal{B}(H^2(X \times \mathbb{R}))$ as a subalgebra of $\mathcal{B}(L^2(X \times \mathbb{R}))$ in
the manner described above, we see that $T_\phi T_\psi - T_{\phi\psi}$ is in $L^1(\tau)$.
By Proposition \ref{Holder}, 
\[
\normone{T_\phi T_\psi - T_{\phi\psi}} = \normone{(PM_\phi Q)(QM_\psi P)} \leq \normtwo{PM_\phi Q}\normtwo{QM_\psi P}.
\]

The next step is to show that $\normtwo{PM_\phi Q} \leq \normCone{\phi}$ for all $\phi$ in $C^1(X, \alpha)$.
From Proposition \ref{opineq}, we have
\[
\normtwo{PM_\phi Q} = \normtwo{\frac{1}{2}P[H, M_\phi]} \leq \frac{1}{2}\normsup{P}\normtwo{[H, M_\phi]} = \frac{1}{2}\normtwo{[H, M_\phi]}.
\]
For $\phi$ in $C^1(X, \alpha)$, define $k_\phi: X \times \mathbb{R} \rightarrow \mathbb{C}$ to be
\[
k_\phi(x, t) = \begin{cases} \frac{\phi(x + t) - \phi(x)}{\pi t} & t \neq 0 \\ {\hskip 10pt} \frac{1}{\pi}\phi^\prime(x) & t = 0.\end{cases}
\]
The proof of Lemma 22.1 in \cite{CMX} states that $\normtwo{[H, M_\phi]} = \normLtwo{k_\phi}$.
\begin{align*}
\normLtwo{k_\phi}^2 &= \frac{1}{\pi^2}\int_X \int_{\mathbb{R}} \abs{\frac{\phi(x + t) - \phi(x)}{t}}^2\, dt\, d\mu(x) \\
= & \frac{1}{\pi^2}\int_X \int_{-\infty}^{-1} \abs{\frac{\phi(x + t) - \phi(x)}{t}}^2\, dt\, d\mu(x) \\
& + \frac{1}{\pi^2}\int_X \int_{1}^{\infty} \abs{\frac{\phi(x + t) - \phi(x)}{t}}^2\, dt\, d\mu(x)\\
& + \frac{1}{\pi^2}\int_X \int_{-1}^1 \abs{\frac{\phi(x + t) - \phi(x)}{t}}^2\, dt\, d\mu(x).
\end{align*}
We compute
\begin{multline*}
\int_{1}^{\infty} \abs{\frac{\phi(x + t) - \phi(x)}{t}}^2\, dt = \int_{1}^{\infty} \frac{\abs{\phi(x + t) - \phi(x)}^2}{t^2}\, dt \\
\leq \int_{1}^{\infty} \frac{(2\normsup{\phi})^2}{t^2}\, dt = 4\normsup{\phi}^2.
\end{multline*}
Similarly,
\[
\int_{-\infty}^{-1} \abs{\frac{\phi(x + t) - \phi(x)}{t}}^2\, dt \leq 4\normsup{\phi}^2.
\]
The Mean Value Theorem gives us
\[
\int_{-1}^1 \abs{\frac{\phi(x + t) - \phi(x)}{t}}^2\, dt \leq \int_{-1}^1\normsup{\phi^\prime}^2\, dt = 2\normsup{\phi^\prime}^2.
\]
Hence, because $\mu$ is a probability measure on $X$,
\begin{align*}
\normLtwo{k_\phi}^2 & \leq \frac{2}{\pi^2}\left(4\normsup{\phi}^2 + \normsup{\phi^\prime}^2\right) \\
& \leq \frac{2}{\pi^2}\left(4\normsup{\phi}^2 + 4\normsup{\phi}\normsup{\phi^\prime} + \normsup{\phi^\prime}^2\right) \\
& = \frac{2}{\pi^2}\left(2\normsup{\phi} + \normsup{\phi^\prime}\right)^2 \\
& \leq \frac{2}{\pi^2}\left(2\normsup{\phi} + 2\normsup{\phi^\prime}\right)^2 \\
& =  \frac{8}{\pi^2}\normCone{\phi}^2 \\
& \leq \normCone{\phi}^2.
\end{align*}
Therefore
\[
 \normtwo{PM_\phi Q} \leq \normLtwo{k_\phi} \leq \normCone{\phi},
\]
as desired.  Finally, for all $\psi$ in $C^1(X, \alpha)$,
\[
\normtwo{QM_\psi P} = \normtwo{(QM_\psi P)^*} = \normtwo{PM_{\bar{\phi}} Q} \leq \normCone{\bar{\psi}} = \normCone{\psi},
\]
whence the theorem follows.
\end{proof}

\begin{definition}\label{SC1T1}
\[
\SC^1(X, \alpha) = \SC(X, \alpha) \cap L^1(\tau)
\]
\[
\mathcal{T}^1(X, \alpha) = \{T_\phi + S : \phi \in C^1(X), S \in \SC^1(X, \alpha)\}
\]
\end{definition}

\begin{proposition}\label{algebra} The set $\mathcal{T}^1(X, \alpha)$ is an algebra, and 
$\SC^1(X, \alpha)$ is an ideal in $\mathcal{T}^1(X, \alpha)$.
\end{proposition}

\begin{proof} The only nonobvious point to check is that $\mathcal{T}^1(X, \alpha)$
is closed under multiplication.  For $T_\phi + S$ and $T_\psi + R$ in $\mathcal{T}^1(X, \alpha)$,
\[
(T_\phi + S)(T_\psi + R) = T_{\phi\psi} + \bigl[(T_\phi T_\psi - T_{\phi\psi}) + T_\phi R + ST_\psi + SR\bigr].
\]
The quantity $T_\phi T_\psi - T_{\phi\psi}$ is in $\SC^1(X, \alpha)$ by Theorem \ref{semi}, and the three other terms enclosed in the 
brackets are in $\SC^1(X, \alpha)$ because $\SC(X, \alpha)$ is an ideal in $\mathcal{T}(X, \alpha)$ and because $L^1(\tau)$ is 
an ideal in $\mathcal{B}(H^2(X \times L^2(\mathbb{R})))$.  Thus $(T_\phi + S)(T_\psi + R)$ is in $\mathcal{T}^1(X, \alpha)$.
\end{proof}

\begin{proposition}\label{SC1} Define $\normSCone{\cdot}$ on $\SC^1(X, \alpha)$ by the formula
\[
\normSCone{S} = \normone{S} + \normop{S}.
\]
Then $\mathcal{SC}^1(X, \alpha)$ is a nonunital Banach algebra under this norm, and
\[
\normSCone{ST_\phi} \leq \normSCone{S}\normCone{\phi},\quad \normSCone{T_\phi S} \leq \normCone{\phi}\normSCone{S}
\]
for all $S$ in $\SC^1(X, \alpha)$ and $\phi$ in $C^1(X, \alpha)$.
\end{proposition}

\begin{proof} 
It is straightforward to check that $\SC^1(X, \alpha)$ is a Banach space under $\normSCone{\cdot}$, and the submultiplicativity of
$\normSCone{\cdot}$ follows easily from Proposition \ref{opineq}.  This proposition also gives us
\[
\normone{ST_\phi} \leq \normone{S}\normop{T_\phi} = \normone{S}\normsup{\phi} \leq \normone{S}\normCone{\phi},
\]
and the submultiplicativity of the operator norm implies that 
\[
\normop{ST_\phi} \leq \normop{S}\normop{T_\phi} = \normop{S}\normsup{\phi} \leq \normop{S}\normCone{\phi},
\]
whence the inequality $\normSCone{ST_\phi} \leq \normSCone{S}\normCone{\phi}$ follows.  The proof of the second
inequality is similar.
\end{proof}

The exact sequence
\[
\xymatrix@C=30pt
{0 \ar[r]^-{} & \SC(X, \alpha) \ar[r]^-{} & \mathcal{T}(X, \alpha) \ar[r]^-{\sigma}  & C(X) \ar[r]^-{} & 0}
\]
restricts to an exact sequence of algebras
\[
\xymatrix@C=30pt
{0 \ar[r]^-{} & \SC^1(X, \alpha) \ar[r]^-{} & \mathcal{T}^1(X, \alpha) \ar[r]^-{\sigma}  & C^1(X, \alpha) \ar[r]^-{} & 0},
\]
and the linear splitting $\xi(\phi) = T_\phi$ restricts as well, implying that every element of $\mathcal{T}^1(X, \alpha)$ can be uniquely
written in the form $T_\phi + S$ for some $\phi$ in $C^1(X, \alpha)$ and $S$ in $\SC^1(X, \alpha)$.
Define $\normoplus{\cdot}$ on $\mathcal{T}^1(X, \alpha)$ by
\[
\normoplus{T_\phi + S} = \normCone{\phi} + \normSCone{S}.
\]
This makes $\mathcal{T}^1(X, \alpha)$ into a Banach space and gives us an exact sequence of Banach spaces.  We will replace
$\normoplus{\cdot}$ with an equivalent Banach algebra norm.  To this end, we first need a lemma.

\begin{lemma}\label{four} For all $T$ and $W$ in $\mathcal{T}^1(X, \alpha)$, 
\[
\normoplus{TW} \leq 4\normoplus{T}\normoplus{W}.
\]
\end{lemma}

\begin{proof}
Write $T = T_\phi + S$ and $W = T_\psi + R$.  Then using Proposition \ref{SC1} and the equation from the proof of Proposition \ref{algebra},
we have
\begin{align*}
\normoplus{TW} &= 
\normoplus{T_{\phi\psi} + \bigl[(T_\phi T_\psi - T_{\phi\psi}) + T_\phi R + ST_\psi + SR\bigr]}\\
&= \normCone{\phi\psi} + \normSCone{(T_\phi T_\psi - T_{\phi\psi}) + T_\phi R + ST_\psi + SR} \\
&\leq \normCone{\phi}\normCone{\psi} + \normSCone{T_\phi T_\psi - T_{\phi\psi}} 
+ \normSCone{T_\phi R} \\
&{\hskip 180pt} + \normSCone{ST_\psi}  + \normSCone{SR}  \\
&\leq \normCone{\phi}\normCone{\psi} + \normSCone{T_\phi T_\psi - T_{\phi\psi}} 
+ \normCone{\phi}\normSCone{R}  \\
&{\hskip 160pt} + \normSCone{S}\normCone{\psi}  + \normSCone{S}\normSCone{R}.
\end{align*}

By Theorem \ref{semi},
\begin{align*}
\normSCone{T_\phi T_\psi - T_{\phi\psi}} &= \normone{T_\phi T_\psi - T_{\phi\psi}} + \normop{T_\phi T_\psi - T_{\phi\psi}} \\
&\leq \normCone{\phi}\normCone{\psi} + \normop{T_\phi}\normop{T_\psi} + \normop{T_{\phi\psi}} \\
&\leq \normCone{\phi}\normCone{\psi} + \normCone{\phi}\normCone{\psi}  + \normCone{\phi\psi} \\
&\leq 3\normCone{\phi}\normCone{\psi}.
\end{align*}

Therefore 
\begin{align*}
\normoplus{TW} &\leq 4\normCone{\phi}\normCone{\psi} + 
\normCone{\phi}\normSCone{R}  + \normSCone{S}\normCone{\psi}  + \normSCone{S}\normSCone{R} \\
&\leq 4\bigl(\normCone{\phi}\normCone{\psi} + 
\normCone{\phi}\normSCone{R}  + \normSCone{S}\normCone{\psi}  + \normSCone{S}\normSCone{R}\bigr) \\
&= 4\normoplus{T}\normoplus{W}.
\end{align*}
\end{proof}

\begin{proposition}\label{BA1} The formula
\[
\normTone{T} := \sup\{\normoplus{TX} : \normoplus{X} \leq 1\}
\]
defines a Banach algebra norm on $\mathcal{T}^1(X, \alpha)$ with the property that
\[
\normoplus{T} \leq \normTone{T} \leq 4\normoplus{T},
\]
and $\normTone{\cdot}$ makes 
\[
\xymatrix@C=30pt
{0 \ar[r]^-{} & \SC^1(X, \alpha) \ar[r]^-{} & \mathcal{T}^1(X, \alpha) \ar[r]^-{\sigma}  & C^1(X) \ar[r]^-{} & 0}
\]
into an exact sequence of Banach algebras.
\end{proposition}

\begin{proof}
We check the inequalities hold; everything else is straightforward to verify.  On one hand,
\[
\normTone{T} \geq \normoplus{TI} = \normoplus{T},
\]
and on the other hand, from Lemma \ref{four},
\[
\normTone{T} = \sup\{\normoplus{TX} : \normoplus{X} \leq 1\} \leq \sup\{4\normoplus{T}\normoplus{X} : \normoplus{X} \leq 1\} = 4\normoplus{T}.
\]
\end{proof}

We record one more result that we will need later in the paper.

\begin{theorem}\label{traceformula} For $T$ and $W$ in $\mathcal{T}^1(X, \alpha)$, the additive commutator $[T, W]$ is in $L^1(\tau)$, and
\[
\tau[T, W] = -\frac{1}{2\pi i}\int_X \sigma(T)^\prime(x)\sigma(W)(x)\, d\mu(x).
\]
\end{theorem}

\begin{proof}
Write $T = T_\phi + S$ and $W = T_\psi + R$. Then
\[
[T, W] =  [T_\phi, T_\psi] + [T_\phi, R] + [S, T_\psi] + [S, R].
\]
The operators $S$ and $R$ are in $L^1(\tau)$, whence $\tau([T, W]) = \tau([T_\phi, T_\psi])$.  Theorem 23.1 in \cite{CMX} states that
\[
\tau[T_\phi, T_\psi] = -\frac{1}{2\pi i}\int_X \phi^\prime(x)\psi(x)\, d\mu(x).
\]
Because $\sigma(T) = \phi$ and $\sigma(W) = \psi$, we obtain the desired result.
\end{proof}

\section{$K$-Theory Preliminaries}\label{K}
Here we will describe the definitions and results we need from \cite{HKS}.  
We will employ various flavors of $K$-theory in this section; 
will denote operator algebra $K$-theory by $\K{top}$, and $\K{alg}$ will be algebraic $K$-theory.  

\begin{definition}\label{relativeBA}
Let $A$ be a unital Banach algebra with norm $\normA{\cdot}$
and let $J$ be a not necessarily closed ideal in $A$. We say that $(A, J)$
is a \emph{relative pair of Banach algebras} if there exists a norm $\normJ{\cdot}$ on $J$ such that
\begin{enumerate}
\item the ideal $J$ is a Banach algebra in the norm $\normJ{\cdot}$;
\item for all $j$ in $J$,
\[
\normJ{j} \leq \normA{j};
\]
\item for all $a$ and $b$ in $A$ and $j$ in $J$,
\[
\normJ{ajb} \leq \normA{a}\normJ{j}\normA{b}.
\]
\end{enumerate}
A morphism between relative Banach pairs $(A, J)$ and $(\widetilde{A}, \widetilde{J})$ is a continuous algebra map 
$\omega: (A, \normA{\cdot}) \rightarrow (\widetilde{A}, \normAt{\cdot})$ that restricts to a continuous map
$\omega_{| J}: (J, \normJ{\cdot}) \rightarrow (\widetilde{J}, \normJt{\cdot})$.
\end{definition}

If $(A, J)$ is a relative pair of Banach algebras, then $(\M(n, A), \M(n, J))$ is also a relative pair of Banach algebras if we define
\[
\normnA{a} = 
\sum_{k, \ell = 1}^n \normA{a_{k\ell}},
\]
and similarly define $\normnJ{j}$ for $j$ in $\M(n, J)$.

For each natural number $n$, define
\begin{align*}
\GL(n, J) &= \{G \in \GL(n, J^+) : G - I_n \in \M(n, J)\} \\
[\GL(n, J), \GL(n, A)] &= \left\{ GHG^{-1}H^{-1} : G \in \GL(n, J), H \in \GL(n, A)\right\}.
\end{align*}
Then $[\GL(n, J), \GL(n, A)]$ is a normal subgroup of $\GL(n, J)$ for each natural number $n$, and
\[
\K{alg}_1(A, J) = \lim_{n\to\infty}\frac{\GL(n, J)}{[\GL(n, J), \GL(n, A)]}.
\]

Let $\R(n, J)$ denote the set of smooth paths $\gamma: [0, 1] \rightarrow \GL(n, J)$ with the property that $\gamma(0) = 1$,
and similarly define $\R(n, A)$.  These sets are groups under pointwise multiplication, and
\[
[\R(n, J), \R(n, A)] = \{\gamma\beta\gamma^{-1}\beta^{-1} : \gamma \in \R(n, J), \beta \in \R(n, A)\}
\]
is a normal subgroup of $\R(n, J)$.   Define an equivalence relation $\sim$ on $\R(n, J)$ by decreeing
that $\gamma_0 \sim  \gamma_1$ if there exists a smooth homotopy $\{\gamma_t\}$ from $\gamma_0$ to
$\gamma_1$ such that $\gamma_t(0) = \gamma_0(0)$ and $\gamma_t(1) = \gamma_0(1)$
for all $0 \leq t \leq 1$.
Let $q$ denote the quotient map from $\R(n, J)$ to the set of equivalence classes of $\sim$, and set
\[
\K{rel}_1(A, J) = \lim_{n\to\infty}\frac{q(\R(n, J))}{q([\R(n, J), \R(n, A)])}.
\]

These four groups fit into an exact sequence
\[
\xymatrix@C=30pt
{\K{top}_0(J) \ar[r]^-{\partial} & \K{rel}_1(A, J) \ar[r]^-{\theta} 
& \K{alg}_1(A, J) \ar[r]^-{p}  & \K{top}_1(J) \ar[r]^-{} & 0,}
\]
with $\theta[\gamma] = [\gamma(1)^{-1}]$ and $p[g] = [g]$.
This exact sequence is a piece of a long exact sequence defined and studied in \cite{Kar} and \cite{CK}.
\vskip 6pt

Many would call the group $\K{alg}_1(A, J)$ 
the first relative algebraic $K$-theory group associated to the pair $(A, J)$.  However, in \cite{Kar} and \cite{CK}
 the authors use the word ``relative'' in the sense that their groups $\K{rel}_*$ measure the 
difference between topological and algebraic $K$-theory.  This double use of the word ``relative''
is unfortunate, but seems to be well established at this point.
\vskip 6pt

Suppose that $J$ admits a continuous linear functional $\tau: J \rightarrow \mathbb{C}$ with the property that
\[
\tau(ja) = \tau(aj)
\]
for all $j$ in $J$ and $a$ in $A$; the authors of \cite{HKS} call this a \emph{hypertrace}.  
Associated to $\tau$ is a group homomorphism $\widetilde{\tau}$ from $\K{rel}_1(A, J)$ to $\mathbb{C}$ that is defined in the following way: 
let $\gamma$ be an element of $\R(1, J)$ and let $[\gamma]$ be the corresponding element of $\K{rel}_1(A, J)$.  Then
\[
\widetilde{\tau}[\gamma] = -\tau\left(\int_0^1\gamma^\prime(t)\gamma(t)^{-1}\, dt\right)
= -\int_0^1\tau\left(\gamma^\prime(t)\gamma(t)^{-1}\right)\, dt.
\]

\begin{definition}\label{HSdet}
Let $\underline{\tau}: \K{top}_0(J) \rightarrow \mathbb{C}$ be the group homomorphism induced by $\tau$.
The \emph{relative de la Harpe-Skandalis determinant} associated to $\tau$ is the group homomorphism 
\[
\widetilde{\det}_\tau: \operatorname{im}(\theta) = \ker(p) \longrightarrow 
\frac{\mathbb{C}}{2\pi i\cdot \operatorname{im}(\underline{\tau})}
\]
that is defined as follows.  Suppose $g$ in $\GL(n, J)$ has the property that its class $[g]$ in $\K{alg}_1(A, J)$
is in the image of $\theta$.  Choose $\beta$ in $\R(n, J)$ so that $\beta(1) = g^{-1}$.  Then
\[
\widetilde{\det}_\tau[g] = \widetilde{\tau}[\beta] + 2\pi i \cdot\operatorname{im}(\underline{\tau}).
\]
\end{definition}

Recall our semifinite von Neumann algebra $\mathcal{N}$ from Section \ref{Toeplitz}.  The trace $\tau$ is a hypertrace,
and the pair $(\mathcal{N}, L^1(\tau))$ is a relative pair of Banach algebras.
The group $\K{top}_1(L^1(\tau))$ is trivial (\cite{HKS}, Lemma 6.4), whence $\ker(p) 
= \K{alg}_1(\mathcal{N}, L^1(\tau))$.  Because the trace of any projection in $\M(n, \mathcal{N})$
is real, we see that $\underline{\tau}(\K{top}_0(L^1(\tau)))$ is contained in $\mathbb{R}$.  Therefore, by expanding
the codomain of $\widetilde{\det}_\tau$,  we have the group homomorphism
\[
\widetilde{\det}_\tau: \K{alg}_1(\mathcal{N}, L^1(\tau)) \rightarrow \mathbb{C}\slash(2\pi i\cdot\mathbb{R}) = 
\mathbb{C}\slash i\mathbb{R}.
\]
Observe that the map $z + i\mathbb{R} \mapsto e^{\Re(z)}$ is a group isomorphism from $\mathbb{C}\slash i\mathbb{R}$
to $(0, \infty)$.  Composing this isomorphism with $\widetilde{\det}_\tau$, we arrive at the following definition.

\begin{definition}\label{FKdet}
The \emph{semifinite Fuglede-Kadison determinant} for $(\mathcal{N}, L^1(\tau))$ is the group homomorphism
\[
{\det}_\tau:  \K{alg}_1(\mathcal{N}, L^1(\tau)) \longrightarrow (0, \infty)
\]
given by the formula
\[
{\det}_\tau [g] = \exp\left(\Re(\widetilde{\det}_\tau[g])\right).
\]
\end{definition}

\begin{proposition}[\cite{HKS}, Proposition 7.5]\label{props}
The semifinite Fuglede-Kadison determinant enjoys the following properties:
\begin{enumerate}
\item ${\det}_\tau[I] = 1$; 
\vskip 6pt
\item ${\det}_\tau[g_1g_2] = {\det}_\tau[g_1]\cdot {\det}_\tau[g_2]$ for all $g_1, g_2$ in $\GL(n, L^1(\tau))$;
\vskip 6pt
\item ${\det}_\tau[hgh^{-1}] = {\det}_\tau[g]$ for all $g$ in $\GL(n, L^1(\tau))$ and $h$ in $\GL(n, \mathcal{N})$.
\end{enumerate}
\end{proposition}

\section{Determinant computations in $\mathcal{T}^1(X, \alpha)$} \label{Tdet}

\begin{proposition}\label{relpair} Let $\mathcal{T}^1(X, \alpha)$ and $\SC^1(X, \alpha)$ be equipped with the Banach algebra norms $\normTone{\cdot}$
and $\normSCone{\cdot}$ respectively.  Then $(\mathcal{T}^1(X, \alpha), \SC^1(X, \alpha))$ is a relative pair of Banach algebras.
\end{proposition}

\begin{proof} For $R$ in $\SC^1(X, \alpha)$, we have $\normSCone{R} \leq \normTone{R}$ by Proposition \ref{BA1}.  We need to prove that 
\[
\normSCone{TRW} \leq \normTone{T}\normSCone{R}\normTone{W}. 
\]
for all $R$ in $\SC^1(X, \alpha)$ and $T$ and $W$ in $\mathcal{T}^1(X, \alpha)$.  It suffices to establish the inequalities 
$\normSCone{TR} \leq \normTone{T}\normSCone{R}$ and $\normSCone{RW} \leq \normSCone{R}\normTone{W}$.  We will verify the first one; 
the proof of the second one is similar.  Thanks to Proposition \ref{BA1}, it is enough to show that $\normSCone{TR} \leq \normoplus{T}\normSCone{R}$.
Write $T = T_\phi +S$.   Then
\begin{align*}
\normSCone{TR} &= \normSCone{(T_\phi + S)R} \\
&= \normop{(T_\phi + S)R} + \normone{(T_\phi + S)R} \\
&\leq (\normop{T_\phi} + \normop{S})\normop{R} + (\normop{T_\phi} + \normop{S})\normone{R} \\
&\leq (\normCone{\phi} + \normop{S})\normSCone{R} \\
&\leq \normoplus{T}\normSCone{R}.
\end{align*}
\end{proof}

\begin{proposition}\label{morphism}
The inclusion map $\iota: (\mathcal{T}^1(X, \alpha), \SC^1(X, \alpha)) \rightarrow (\mathcal{N}, L^1(\tau))$ is a morphism of relative Banach pairs.
\end{proposition}

\begin{proof} Take $T_\phi + S$ in $\mathcal{T}^1(X, \alpha)$.  
\[
\normop{\iota(T_\phi + S)} \leq \normop{T_\phi} + \normop{S} = \normsup{\phi} + \normop{S} \leq \normCone{\phi} + \normSCone{S} = \normoplus{T_\phi + S}.
\]
The norms $\normoplus{\cdot}$ and $\normTone{\cdot}$ on $\mathcal{T}^1(X, \alpha)$ are equivalent, so $\iota: \mathcal{T}^1(X, \alpha) \rightarrow \mathcal{N}$
is continuous.  The restriction of $\iota$ to $\SC^1(X, \alpha)$ is also continuous, because 
\[
\normone{\iota(S)} = \normone{S} \leq \normSCone{S}.
\]
\end{proof}

Define a map $d: \GL(n, \SC^1(X, \alpha)) \rightarrow (0, \infty)$ in the following manner.
Suppose that $Q$ is an element in $\GL(n, \SC^1(X, \alpha))$.  Then $Q$ determines 
an element of $[Q]$ of $\K{alg}_1(\mathcal{T}^1(X, \alpha), \SC^1(X, \alpha))$ and 
$[\iota(Q)]$ is in $\K{alg}_1(\mathcal{N}, L^1(\tau))$.  Thus we can set
\[
d(Q) = {\det}_\tau(\iota[Q]).
\]
Thanks to Proposition \ref{props}, we see that the map $d$ has the following properties:
\begin{enumerate}
\item $d(I) = 1$;
\vskip 6pt
\item $d(Q_1Q_2) = d(Q_1)d(Q_2)$ for $Q_1$ and $Q_2$ in $\GL(n, \SC^1(X, \alpha))$;
\vskip 6pt
\item $d(GQG^{-1}) = d(Q)$ for $Q$ in $\GL(n, \SC^1(X, \alpha))$ and $G$ in $\GL(n, \mathcal{T}^1(X, \alpha))$.
\end{enumerate}

Therefore $d$ can be considered to be a determinant function.  
In this section we derive a formula for computing $d$ on certain multiplicative commutators.  While many of the results we prove in the section
hold for arbitrary $n$, henceforth we will restrict to the situation when $n = 1$.

Let $G$ and $H$ be elements of $\GL(1, \mathcal{T}^1(X, \alpha))$.  Then 
\[
\sigma(GHG^{-1}H^{-1}) = \sigma(G)\sigma(H)\sigma(G)^{-1}\sigma(H)^{-1} = 1,
\]
whence $GHG^{-1}H^{-1}$ is an element of $\GL(1, \SC^1(X, \alpha))$.

\begin{proposition}\label{multcomm}
The value of $d(GHG^{-1}H^{-1})$ depends only on $\sigma(G)$ and $\sigma(H)$.  
\end{proposition}

\begin{proof}
Suppose $\sigma(\widetilde{G}) = \sigma(G)$ for some element $\widetilde{G}$ of $\GL(1, \mathcal{T}^1(X, \alpha))$.
Then the multiplicativity and similarity invariance of $d$ yield the following string of equalities:
\begin{align*}
d(\widetilde{G}H\widetilde{G}^{-1}H^{-1}) 
&= d(G^{-1}(\widetilde{G}H\widetilde{G}^{-1}H^{-1})G) \\
&= d(G^{-1}\widetilde{G}) d(H\widetilde{G}^{-1}H^{-1}G) \\
&= d(G^{-1}\widetilde{G}) d(H(\widetilde{G}^{-1}H^{-1}GH)H^{-1}) \\
&= d(G^{-1}\widetilde{G}) d(\widetilde{G}^{-1}H^{-1}GH) \\
&= d(G^{-1}\widetilde{G}) d(G(\widetilde{G}^{-1}H^{-1}GH)G^{-1}) \\
&= d(G^{-1}\widetilde{G}) d(G\widetilde{G}^{-1}) d(H^{-1}GHG^{-1}) \\                             
&= d(\widetilde{G}(G^{-1}\widetilde{G})\widetilde{G}^{-1}) d(G\widetilde{G}^{-1}) d(H(H^{-1}GHG^{-1})H^{-1}) \\   
&= d(\widetilde{G}G^{-1}) d(G\widetilde{G}^{-1})  d(GHG^{-1}H^{-1}) \\ 
&= d(I) d(GHG^{-1}H^{-1}) \\                                 
&= d(GHG^{-1}H^{-1}).
\end{align*}
A similar argument shows that if we have $\sigma(\widetilde{H}) = \sigma(H)$ for some element $\widetilde{H}$ of 
$\GL(1, \mathcal{T}^1(X, \alpha))$, then $d(G\widetilde{H}G^{-1}\widetilde{H}^{-1}) = d(GHG^{-1}H^{-1})$.
\end{proof}

In the case where $G = e^T$ and $H = e^W$ for $T$ and $W$ in $\mathcal{T}^1(X, \alpha)$, we can write down a formula for
 $d(GHG^{-1}H^{-1})$ in terms of $\sigma(T)$ and $\sigma(W)$.   
We begin with a result that seems to be well known, but I could not find a proof in print.  The proof here is adapted from an argument
presented in \cite{BS}.  We only need the result for elements of $\mathcal{T}^1(X, \alpha)$, but the same proof works in any Banach
algebra, so we will prove this more general result.

\begin{lemma}\label{brightsun}
Let $\mathcal{A}$ be a Banach algebra.  For all $T$ and $W$ in $\mathcal{A}$,
\[
e^TWe^{-T} = W + [T, W] + \frac{1}{2!}[T, [T, W]] + \frac{1}{3!}[T, [T, [T, W]]] + \cdots
\]
\end{lemma}

\begin{proof}
Define a sequence $\{W_n\}$ by setting $W_0 = W$ and $W_{n+1} = [T, W_n]$ for $n > 0$.  Then the right hand side of our equation is
\[
\sum_{n=0}^\infty\frac{W_n}{n!}.
\]
An induction argument shows that $\norm{W_n} \leq 2^n\norm{T}^n\norm{W}$ 
for all $n$, so we see that this infinite series converges by comparing it to the Maclaurin series for $e^x$.

 Next, write out the left hand side in terms of the Maclaurin series for the exponential:

\begin{align*}
e^TWe^{-T} &= \left(\sum_{\ell=0}^\infty\frac{T^\ell}{\ell!}\right) W \left(\sum_{k=0}^\infty\frac{(-T)^k}{k!}\right) \\
&= \sum_{\ell=0}^\infty\sum_{k=0}^\infty\frac{(-1)^k}{\ell!\, k!}T^\ell W T^k \\
&= \sum_{n=0}^\infty\sum_{k=0}^n\frac{(-1)^k}{k!\, (n - k)!}T^{n - k}W t^k \\
&= \sum_{n=0}^\infty\sum_{k=0}^n\frac{(-1)^k}{n!}\frac{n!}{k!\, (n - k)!}T^{n - k}W T^k \\
&= \sum_{n=0}^\infty\frac{1}{n!}\left(\sum_{k=0}^n (-1)^k\binom{n}{k} T^{n - k}W T^k\right);
\end{align*}
starting in the third line we have $n = \ell + k$.
To complete the proof, we use an induction argument to show that
\[
W_n = \sum_{k=0}^n (-1)^k\binom{n}{k} T^{n - k}W T^k
\]
for all nonnegative integers $n$.  The formula is certainly correct for $n = 0$.  Suppose the formula holds for $n$.  Then, using a standard
binomial identity we compute that

\begin{align*}
W_{n+1} &= [T, W_n] \\
&= \sum_{k=0}^n (-1)^k\binom{n}{k}\left[T, T^{n - k}W T^k\right] \\
&= \sum_{k=0}^n (-1)^k\binom{n}{k}\left(T^{n - k + 1}W T^k - T^{n - k}WT^{k + 1}\right)\\
&= T^{n + 1}W +  \sum_{k=1}^n(-1)^k\binom{n}{k} T^{n - k + 1}W T^k \\
&{\hskip 80pt}- \sum_{k=0}^{n-1}(-1)^k\binom{n}{k} T^{n - k}W T^{k+1} - (-1)^nWT^{n+1}\\
&= T^{n + 1}W +  \sum_{k=1}^n(-1)^k\binom{n}{k} T^{n - k + 1}W T^k \\
&{\hskip 80pt}- \sum_{k=1}^n(-1)^{k-1}\binom{n}{k-1} T^{n - k+1}W T^k + (-1)^{n+1}WT^{n+1}\\
&= T^{n + 1}W + \sum_{k=1}^n(-1)^k\left(\binom{n}{k-1} + \binom{n}{k}\right) T^{n - k + 1}WT^k\\
&{\hskip 230pt} + (-1)^{n+1}WT^{n+1}\\    
&= T^{n + 1}W + \sum_{k=1}^n(-1)^k\binom{n+1}{k} T^{n - k + 1}W T^k + (-1)^{n+1}WT^{n+1}\\      
&= \sum_{k=0}^{n+1} (-1)^k\binom{n+1}{k} T^{n +1 - k}W T^k.           
\end{align*} 
\end{proof}

\begin{proposition}\label{dettr}
For all $T$ and $W$ in $\mathcal{T}^1(X, \alpha)$,
\[
\widetilde{\det}_\tau\left[\iota(e^Te^We^{-T}e^{-W})\right] = \tau(TW - WT) + i\mathbb{R}.
\]
\end{proposition}

\begin{proof}
Define $\beta \in R(1, L^1(\tau))$ by the formula 
\[
\beta(t) =  e^{tW} e^{T} e^{-tW} e^{-T}.
\]
We compute
\begin{align*}
\beta^\prime(t)\beta(t)^{-1} &= 
\left(We^{tW}e^Te^{-tW}e^{-T} - e^{tW}e^TWe^{-tW}e^{-T}\right)e^Te^{tW}e^{-T}e^{-tW} \\
&= W - e^{tW}e^TWe^{-T}e^{-tW}.
\end{align*}

Because $\tau$ is similarity invariant and because $e^{tW}$ commutes with $W$, we see that
\begin{align*}
\widetilde{\tau}[\beta] &= -\int_0^1\tau\left(\beta^\prime(t)\beta(t)^{-1}\right) dt \\
&= -\int_0^1\tau\left(W - e^{tW}e^TWe^{-T}e^{-tW}\right) dt\\
&= -\int_0^1\tau\left(e^{tW}We^{-tW} - e^{tW}e^TWe^{-T}e^{-tW}\right) dt\\
&= -\int_0^1\tau\left(W - e^TWe^{-T}\right) dt\\
&= -\tau\left(W - e^TWe^{-T}\right).
\end{align*}
Use Lemma \ref{brightsun} to expand $W - e^TWe^{-T}$:
\[
W - e^TWe^{-T} = -\left([T, W] + \frac{1}{2!}[T, [T, W]] + \frac{1}{3!}[T, [T, [T, W]]] + \cdots\right).
\]
From Lemma \ref{brightsun}, we know that the right side of this equation converges in the norm on $\mathcal{T}^1(X, \alpha)$.
Each summand is in $\SC^1(X, \alpha)$, the norm on $\mathcal{T}^1(X, \alpha)$
dominates the $L^1(\tau)$ norm, and $\tau$ is continuous in the $L^1(\tau)$ norm.  Therefore 
\[
-\tau\left(W - e^TWe^{-T}\right) =  
\tau\left([T, W]\right) + \frac{1}{2!}\tau\left([T, [T, W]]\right) + \frac{1}{3!}\tau\left([T, [T, [T, W]]]\right) + \cdots .
\]
Because $\tau$ is a hypertrace,
all of the terms on the right side vanish except the first one, and thus
\[
\widetilde{\det}_\tau\left[\iota(e^Te^We^{-T}e^{-W})\right] = \widetilde{\tau}[\beta] + i\mathbb{R} =  \tau(TW - WT) + i\mathbb{R}.
\]
\end{proof}

\begin{theorem}\label{dformula}
Let $T$ and $W$ be elements of $\mathcal{T}^1(X, \alpha)$.  Then
\[
d(e^Te^We^{-T}e^{-W}) =  \exp\left(\frac{1}{2\pi}\int_X{\kern-0.5ex}\Im\bigl(\sigma(T)^\prime(x)\sigma(W)(x)\bigr)\, d\mu(x)\right).
\]
\end{theorem}

\begin{proof}
From Theorem \ref{traceformula}, Proposition \ref{dettr}, and the definition of ${\det}_\tau$, we have
\begin{align*}
d(e^Te^We^{-T}e^{-W}) &= {\det}_\tau\bigl(\iota[e^Te^We^{-T}e^{-W}]\bigr) \\
&= \exp\bigl(\Re(\tau(TW - WT))\bigr) \\
&= \exp\left(\Re\left(-\frac{1}{2\pi i}\int_X{\kern-0.5ex}\sigma(T)^\prime(x)\sigma(W)(x)\, d\mu(x)\right)\right)\\
&= \exp\left(\frac{1}{2\pi}\int_X{\kern-0.5ex}\Im\bigl(\sigma(T)^\prime(x)\sigma(W)(x)\bigr)\, d\mu(x)\right).
\end{align*}
\end{proof}

We can use Theorem \ref{dformula} and the long exact sequence in algebraic $K$-theory to construct a homomorphism from $\K{alg}_2(C^1(X, \alpha))$ to $\mathbb{R}$. 
The reader is directed to either \cite{M} and \cite{R} for every readable sources about algebraic $K$-theory.  In particular, Theorem 13.20 in \cite{M} and Theorem 4.3.1 in \cite{R} discuss the connecting map $\partial$ used in the next definition.

\begin{definition}\label{Deltadef} 
Let $\partial: \K{alg}_2(C^1(X, \alpha)) \rightarrow \K{alg}_1(\mathcal{T}^1(X, \alpha), \SC^1(X, \alpha))$ be the connecting map from the long exact sequence in algebraic
$K$-theory associated to the short exact sequence
\[
\xymatrix@C=30pt
{0 \ar[r]^-{} & \SC^1(X, \alpha) \ar[r]^-{} & \mathcal{T}^1(X, \alpha) \ar[r]^-{\sigma}  & C^1(X, \alpha) \ar[r]^-{} & 0}.
\]
Define the group homomorphism $\Delta: \K{alg}_2(C^1(X, \alpha)) \rightarrow (0, \infty)$ to be the composition
\[
\xymatrix@C=10pt
{\K{alg}_2(C^1(X, \alpha)) \ar[r]^-{\partial} & \K{alg}_1(\mathcal{T}^1(X, \alpha), \SC^1(X, \alpha)) \ar[r]^-{\iota_*}
& \K{alg}_1(\mathcal{N}, L^1(\tau)) \ar[r]^-{\det_\tau} & (0, \infty)}.
\]
\end{definition}

\begin{proposition}\label{specialcase}
Let $\phi$ and $\psi$ be in $C^1(X, \alpha)$, and let $\{e^\phi, e^\psi\}$ denote the Steinberg symbol of $e^\phi$ and $e^\psi$.  Then 
\[
\partial \{e^\phi, e^\psi\} = \left[e^{T_\phi} e^{T_\psi} e^{-T_\phi} e^{-T_\psi}\right].
\]
\end{proposition}

\begin{proof}
Note that $e^\phi$ and $e^\psi$ are invertible elements of
$C^1(X, \alpha)$ that lift to invertible elements $e^{T_\phi}$ and $e^{T_\psi}$ in $\mathcal{T}^1(X, \alpha)$.  From Proposition 4.4.22 in \cite{R}, we can write 
\[
\{e^\phi, e^\psi\} = h_{12}(e^\phi e^\psi)^{-1}h_{12}(e^\phi)h_{12}(e^\psi),
\]
which lifts to
\[
h_{12}\left(e^{T_\phi}e^{T_\psi}\right)^{-1}h_{12}\left(e^{T_\phi}\right) h_{12}\left(e^{T_\psi}\right)
\]
in the Steinberg group $\operatorname{St}(\mathcal{T}^1(X, \alpha))$.  Applying the formula for the algebraic $K$-theory connecting map, we arrive at the following matrix that
represents an element of the group $\K{alg}_1(\mathcal{T}^1(X, \alpha), \SC^1(X, \alpha))$:
\begin{multline*}
\begin{pmatrix}
\left(e^{T_\phi}e^{T_\psi}\right)^{-1} & 0 \\ 0 & e^{T_\phi}e^{T_\psi}
\end{pmatrix}
\begin{pmatrix}
e^{T_\phi} & 0 \\ 0 & \left(e^{T_\phi}\right)^{-1}
\end{pmatrix}
\begin{pmatrix}
e^{T_\psi} & 0 \\ 0 & \left(e^{T_\psi}\right)^{-1}
\end{pmatrix} = \\
\begin{pmatrix} 
I & 0 \\ 0 & e^{T_\phi} e^{T_\psi} e^{-T_\phi} e^{-T_\psi}
\end{pmatrix}.
\end{multline*}

It is straightforward to show that this matrix determines the same element of $\K{alg}_1(\mathcal{T}^1(X, \alpha), \SC^1(X, \alpha))$ as 
$e^{T_\phi} e^{T_\psi} e^{-T_\phi} e^{-T_\psi}$.   
\end{proof}

\begin{theorem}\label{main}
Let $\phi$ and $\psi$ be in $C^1(X, \alpha)$.  Then
\[
\Delta\left(\{e^\phi, e^\psi\}\right) =  \exp\left(\frac{1}{2\pi}\int_X{\kern-0.5ex}\Im\left(\phi^\prime(x)\psi(x)\right)\, d\mu(x)\right).
\]
\end{theorem}

\begin{proof}
From Theorem \ref{dformula} and the definition of $\Delta$, we have
\begin{align*}
\Delta\left(\{e^\phi, e^\psi\}\right) &= {\det}_\tau\left(\iota\left[e^{T_\phi} e^{T_\psi} e^{-T_\phi} e^{-T_\psi}\right]\right) \\
&= \exp\left(\frac{1}{2\pi}\int_X{\kern-0.5ex}\Im\bigl(\sigma(T)^\prime(x)\sigma(W)(x)\, d\mu(x)\bigr)\right).
\end{align*}
\end{proof}

\end{document}